\documentclass[10pt]{article}%
\usepackage{amsmath}
\usepackage{amsfonts}
\usepackage{amssymb}
\usepackage{graphicx}%
\usepackage{mathrsfs}
\usepackage[body={15.5cm,21cm}, top=3cm]{geometry}
\DeclareMathOperator*{\esssup}{ess\,sup}
\providecommand{\U}[1]{\protect \rule{.1in}{.1in}}
\newtheorem{theorem}{Theorem}[section]

\newtheorem{definition}[theorem]{Definition}

\newtheorem{lemma}[theorem]{Lemma}

\newtheorem{remark}[theorem]{Remark}

\newenvironment{proof}[1][Proof]{\noindent \textbf{#1.} }{\  \rule{0.5em}{0.5em}}
\begin{document}

\title{Mean-field backward stochastic differential equations with mean reflection and nonlinear resistance}
\author{Peng Luo \thanks{School of Mathematical Sciences, Shanghai Jiao Tong University, Shanghai 200240, China (peng.luo@sjtu.edu.cn). Financial support
from the National Natural Science Foundation of China (Grant No. 12101400) is gratefully acknowledged.}}
\date{}

\maketitle

\begin{abstract}
The present paper is devoted to the study of the well-posedness of mean field BSDEs with mean reflection and nonlinear resistance. By the contraction mapping argument, we first prove that the mean-field BSDE with mean reflection and nonlinear resistance admits a unique deterministic flat local solution on a small time interval. Moreover, we build the global solution by introducing a two-step approach, which is a combination of stitching method and fixed point method. We further provide an application to the super-hedging problem with risk constraint. 
\end{abstract}


\textbf{MSC-classification}: 60H10, 60H30, 91G10.
\section{Introduction}
The nonlinear Backward Stochastic Differential Equation (BSDE) of the following form was first introduced by Pardoux and Peng \cite{PP}:
\begin{align}\label{1}
 Y_t=\xi+\int^T_t f(s,Y_s,Z_s)ds-\int^T_tZ_sdB_s,\ \ \ \  \forall t\in[0,T],
\end{align}
whose solution  consists of  an adapted pair of processes $(Y,Z)$.
Pardoux and Peng have obtained  the existence and uniqueness theorem for the BSDE \eqref{1} when the generator $f$ is uniformly Lipschitz and
the terminal value $\xi$ is square integrable. The solvability of scalar-valued quadratic BSDEs was first established by
Kobylanski \cite{K1} via a PDE-based method under the  boundedness assumption of the terminal value.
Subsequently, Briand and Hu \cite{BH2,BH3} have extended the existence result to the case of unbounded terminal values with exponential moments and have studied the uniqueness whenever the generator is convex (or concave).

Recently, motivated by super-hedging of claims under running risk management constraints,  Briand, Elie and Hu \cite{BH} have formulated a new type of BSDE with constraints,
 which is called the BSDE with mean reflection. In their framework, the solution $Y$  is required to satisfy the following type of mean reflection constraint:
  \[
  \mathbb{E}[\ell(t,Y_t)]\geq 0, \ \ \ \forall t\in[0,T],
  \]
 where the running loss function $(\ell(t,\cdot))_{0\leq t\leq T}$ is a collection of  (possibly random) non-decreasing real-valued map. This  type  of reflected equation is also closely related to interacting particles systems,  see, e.g., Briand, Chaudru de Raynal,  Guillin and  Labart \cite{BC}. In order to establish the well-posedness of BSDEs with mean reflection, in \cite{BH} the authors have introduced the notion of deterministic flat solution and obtain the existence and uniqueness of a deterministic flat solution when the generator is Lipchitz in $(y,z)$ and terminal condition is square integrable. Hibon et al. \cite{HHLLW} further studied the case with quadratic generator and bounded terminal condition.

The main purpose of this paper is to study the following type of BSDE with mean reflection:
\begin{align*}
\begin{cases}
&Y_t=\xi+\int^T_t f(s,Y_s,\mathbb{P}_{Y_s},Z_s,\mathbb{P}_{Z_s},G_s(K))ds-\int^T_t Z_s dB_s+K_T-K_t;\\
& \ \ \ \ \ \ \ \ \ \ \ \ \ \ \ \ \ \ \mathbb{E}[\ell(t,Y_t)]\geq 0,
\end{cases} \  \  \forall t\in[0,T],
\end{align*}
where $\mathbb{P}_{Y_s}$ and $\mathbb{P}_{Z_s}$ are the laws of $Y_s$ and $Z_s$ respectively and $G(K)$ is the resistance term \footnote{Reflected BSDE with nonlinear resistance was first studied in \cite{QX}. Following their work, we call $G(K)$ the resistance term.}. This equation is called mean-field BSDE with mean reflection and nonlinear resistance. Our work is motivated by super-hedging problem with risk constraint and the study of interacting particles system. When the mean reflection constraint and resistance term vanish, it reduces to a standard mean field BSDE. For the study of mean-field BSDEs, we refer the reader to \cite{BDLP,BLP,HHT}.

As in \cite{BH,HHLLW}, we consider the existence and uniqueness of a deterministic flat solution of the above mean-field BSDE with mean reflection and nonlinear resistance. Relying on a contraction mapping argument, we first show that there exists a unique deterministic flat solution on a small time horizon for both the case where the generator is Lipschitz and terminal condition is square integrable and the case where the generator is quadratic in $Z_s$, subquadratic in $\mathbb{P}_{Z_s}$ and terminal condition is bounded. However, the nonlinear resistence term brings additional difficulty to study the global solution. To overcome this difficulty, we introduce a two-step approach which is a combination of stitching method and fixed point method. One of the main ingredients is the observation that the maximal length of the time interval on which the mapping is contractive is a fixed constant in the Lipschitz case and depends only on the bound of the component $Y$ in the quadratic case. Therefore for a given resistance term, we obtain a global solution on the whole time interval in the Lipschitz case by stitching the local ones. We further show that the component $Y$ has a uniform estimate under additional assumptions in the quadratic case. Once again, a global solution on the whole time interval in the quadratic case can be established by stitching local solutions for a given resistance term. Another main ingredient is to apply the fixed point method with respect to the resistance term. A contraction mapping could be constructed thanks to the global solution built in the first step and the solution is deterministic flat. In conclusion, there exists a unique deterministic flat solution for the mean-field BSDE with nonlinear resistance both in the Lipschitz and quadratic case.

The remainder of the paper is organized as follows. In Section 2, we first introduce mean-field BSDEs with mean reflection and nonlinear resistance, and state our main assumptions. When the generator is Lipschitz and terminal condition is square integrable, the existence and uniqueness of deterministic flat solution is first obtained. We further investigate the case when the generator is quadratic and terminal condition is bounded. An application to the super-hedging probem with risk constraint is presented in Section 3.

\subsubsection*{Notation.}
We introduce the notations, which will be used throughout this paper.   For  each Euclidian space, we  denote by $\langle\cdot,\cdot \rangle$  and  $|\cdot|$
 its scalar product and the associated norm, respectively.
Then consider  a  finite time horizon $[0,T]$ and  a complete probability space $(\Omega,\mathscr{F},\mathbb{P})$,  on which $B=(B_t)_{0\leq t\leq T}$ is a  standard $d$-dimensional Brownian motion.  Let $(\mathscr{F}_t)_{0\leq t\leq T}$ be the  natural filtration generated by $B$ augmented with the family $\mathscr{N}^\mathbb{P}$ of  $\mathbb{P}$-null sets of $\mathscr{F}$. $\mathcal{P}$ denotes the sigma algebra of progressively measurable sets of $[0,T]\times\Omega$,
   $\mathcal{B}(\mathbb{R})$ and $\mathcal{B}(\mathbb{R}^d)$ are the Borel algebras on $\mathbb{R}$ and $\mathbb{R}^d$, respectively. $\mathcal{C}_{T}(\mathbb{R})$ denotes the set $C([0,T];\mathbb{R})$ of continuous functions from $[0,T]$ to $\mathbb{R}$. For $x\in\mathbb{R}^d$, $\delta_x$ denotes the Dirac measure at $x$. 
  Finally for all $t\in[0,T]$
 we consider the following Banach spaces:
\begin{description}
\item[$\bullet$] $\mathcal{L}^{2}(\mathscr{F}_t)$ is the space of  real valued $\mathscr{F}_t$-measurable random variables $Y$
satisfying
\begin{align*}
\|Y\|_{\mathcal{L}^{2}}=\mathbb{E}[|Y|^2]^{\frac{1}{2}}<\infty;
\end{align*}
\item[$\bullet$]  $\mathcal{L}^{\infty}(\mathscr{F}_t)$  is the space of  real valued $\mathscr{F}_t$-measurable random variables $Y$
satisfying
\begin{align*}
\|Y\|_{\mathcal{L}^{\infty}}=\esssup\limits_{\omega}|Y(\omega)|<\infty;\end{align*}
\item[$\bullet$]  $\mathcal{S}^{2}$  is the space of  real valued  progressively measurable continuous  processes $Y$ satisfying
\begin{align*}
\|Y\|_{\mathcal{S}^{2}}=\mathbb{E}\left[\sup_{0\leq t\leq T}|Y_t|^2\right]^{\frac{1}{2}}<\infty;
\end{align*}
\item[$\bullet$]  $\mathcal{A}^2_D$ is the closed subset of $\mathcal{S}^{2}$ consisting of deterministic non-decreasing processes $K = (K_t)_{0\leq t\leq T}$ starting from the origin;
\item[$\bullet$]
 $\mathcal{H}^2$  is the space of  all  progressively measurable  processes $Z$ taking values in $\mathbb{R}^d$ such that
\begin{align*}
\|Z\|_{\mathcal{H}^2} =\mathbb{E}\left[\int_0^T|Z_t|^2dt\right]^{\frac{1}{2}}< \infty;
\end{align*}
\item[$\bullet$]  $\mathcal{S}^{\infty}$  is the space of  real valued  progressively measurable continuous  processes $Y$ satisfying
\begin{align*}
\|Y\|_{\mathcal{S}^{\infty}}=\esssup\limits_{(t,\omega)}|Y(t,\omega)|<\infty;
\end{align*}
\item[$\bullet$]  $\mathcal{A}^{\infty}_D$ is the closed subset of $\mathcal{S}^{\infty}$ consisting of deterministic non-decreasing processes $K = (K_t)_{0\leq t\leq T}$ starting from the origin;
\item[$\bullet$]
 $BMO$  is the space of  all  progressively measurable  processes $Z$ taking values in $\mathbb{R}^d$ such that
\begin{align*}
\|Z\|_{BMO} =\sup\limits_{\tau\in\mathcal{T}}\left\|\mathbb{E}_{\tau}\left[\int^T_{\tau}|Z_s|^2ds\right]\right\|_{\mathcal{L}^{\infty}}^{1/2} < \infty,
\end{align*}
where $\mathcal{T}$ denotes the set of all $[0,T]$-valued stopping times $\tau$
and $\mathbb{E}_{\tau}$ is the conditional expectation with respect to $\mathscr{F}_{\tau}$.
 \end{description}
 For each $Z\in BMO$, we set
 \[
  \mathscr{Exp}(Z\cdot B)_T=\exp\left(\int^T_0 Z_s dB_s-\frac{1}{2}\int^T_0|Z_s|^2ds\right),
 \]
which is a martingale  by
 \cite{K}. Thus it follows from Girsanov's theorem that \newline $(B_t-\int_{}^tZ_sds)_{0\leq t\leq T}$
is a Brownian motion under the equivalent probability measure $\mathscr{Exp} (Z\cdot B)_{T} d\mathbb{P}$.  We further denote by $\mathcal{S}^2_{[a,b]}$, $\mathcal{H}^2_{[a,b]}$, $\mathcal{A}^2_{D,[a,b]}$ $\mathcal{S}^{\infty}_{[a,b]}$, $\mathcal{A}^{\infty}_{D,[a,b]}$ and $BMO_{[a,b]}$ the corresponding spaces for the stochastic processes have time indexes on $[a,b]$. For $p\geq 1$, let $\mathcal{P}_p(\mathbb{R}^d)$ be the set of all probability measures $\mu$ on $\left(\mathbb{R}^d,\mathcal{B}(\mathbb{R}^d)\right)$ satisfying $\int_{\mathbb{R}^d}|x|^p\mu(dx)<\infty$. Moreover, $\mathcal{P}_p(\mathbb{R}^d)$ is endowed with the $p$-Wasserstein metric: for $\mu,\nu\in\mathcal{P}_p(\mathbb{R}^d)$,
\begin{equation*}
  \mathcal{W}_p\left(\mu,\nu\right)=\inf\left\{\left(\int_{\mathbb{R}^d\times\mathbb{R}^d}|x-y|^p\pi(dxdy)\big|\pi\in\mathcal{P}_2(\mathbb{R}^{2d}),\pi(\cdot\times\mathbb{R}^d)=\mu,\pi(\mathbb{R}^d\times\cdot)=\nu\right)\right\}.
\end{equation*}
\section{Mean field BSDEs with mean reflection and nonlinear resistance}
In this paper, we  consider the following type of BSDE with mean reflection:
\begin{align}\label{my1}
\begin{cases}
&Y_t=\xi+\int^T_t f(s,Y_s,\mathbb{P}_{Y_s},Z_s,\mathbb{P}_{Z_s},G_s(K))ds-\int^T_t Z_s dB_s+K_T-K_t;\\
& \ \ \ \ \ \ \ \ \ \ \ \ \ \ \ \ \ \ \mathbb{E}[\ell(t,Y_t)]\geq 0,
\end{cases} \  \  \forall t\in[0,T],
\end{align}
where $\mathbb{P}_{Y_s}$ and $\mathbb{P}_{Z_s}$ are the laws of $Y_s$ and $Z_s$ respectively and $G(K)$ is the resistance term. The above equation is called mean-field BSDE with mean reflection and nonlinear resistance. Its parameters are the terminal condition $\xi$, the generator (or driver) $f$, resistance function $G$ as well as the running loss function $\ell$. BSDE with mean reflection was first introduced in \cite{BH}, where the authors have discussed such equation under the standard Lipschitz condition on the generator and the square integrability assumption on terminal condition. Quadratic BSDE with mean reflection and bounded terminal condition was investigated in \cite{HHLLW}.

Our work is motivated by super-hedging problem with risk constraint and the study of interacting particles system. In fact, mean-field BSDE with mean reflection and nonlinear resistance \eqref{my1} arises (with a slight disction) naturally from super-hedging problem with risk constraint, we refer the reader to Section 3 for more details. When the resistance tem vanishes, \eqref{my1} can be viewed as the limiting system of particle systems interacting through both the driver and reflection (see \cite{BH0,BDLP}).

In the sequel, we  study the existence and uniqueness theorem of equation \eqref{my1}. These parameters are supposed to satisfy the following standard running assumptions:
 \begin{itemize}
 \item[($H_\xi$)] The terminal condition $\xi$ is an square-integrable  ${\mathcal{ F}}_T$-measurable random variable such that  \[\mathbb{E}[\ell(T,\xi)]\geq 0.\]
 \item[($H^q_\xi$)] The terminal condition $\xi$ is an  ${\mathcal{ F}}_T$-measurable random variable bounded by $L>0$ such that  \[\mathbb{E}[\ell(T,\xi)]\geq 0.\]
 \item[($H_f$)]  The driver $f:[0,T]\times\Omega\times\mathbb{R}\times \mathcal{P}_2(\mathbb{R})\times\mathbb{R}^d\times\mathcal{P}_2(\mathbb{R}^d)\times\mathbb{R} \rightarrow\mathbb{R}$ is a $\mathcal{P}\times \mathcal{B}(\mathbb{R})\times \mathcal{B}(\mathcal{P}_2(\mathbb{R}))\times\mathcal{B}(\mathbb{R}^d)\times\mathcal{B}(\mathcal{P}_2(\mathbb{R}^d))\times \mathcal{B}(\mathbb{R})$-measurable map, and there exists $\lambda>0$ such that for all $t\in[0,T]$, $y,\bar{y},k,\bar{k}\in \mathbb{R}$, $\mu,\bar{\mu}\in\mathcal{P}_2(\mathbb{R})$ and $z,\bar{z}\in\mathbb{R}^d$, $\nu,\bar{\nu}\in\mathcal{P}_2(\mathbb{R}^d)$,
	\begin{equation*}
		| f(t,y,\mu,z,\nu,k) - f(t,\bar{y},\bar{\mu},\bar{z},\bar{\nu},\bar{k})| \leq \lambda( |y-\bar{y}| +\mathcal{W}_2(\mu,\bar{\mu}) + |z-\bar{z}| +\mathcal{W}_2(\nu,\bar{\nu}) + |k-\bar{k}|)
	\end{equation*}
and
\begin{equation*}
E\left[\int_0^T|f(t,0,\delta_0,0,\delta_0,0)|^2dt\right]<\infty.
\end{equation*}
 \item[($H^q_f$)]  he driver $f:[0,T]\times\Omega\times\mathbb{R}\times \mathcal{P}_2(\mathbb{R})\times\mathbb{R}^d\times\mathcal{P}_2(\mathbb{R}^d)\times\mathbb{R} \rightarrow\mathbb{R}$ is a $\mathcal{P}\times \mathcal{B}(\mathbb{R})\times \mathcal{B}(\mathcal{P}_2(\mathbb{R}))\times\mathcal{B}(\mathbb{R}^d)\times\mathcal{B}(\mathcal{P}_2(\mathbb{R}^d))\times \mathcal{B}(\mathbb{R})$-measurable map such that
     \begin{enumerate}
   \item[(1)] For each $t\in[0,T]$, $f(t,0,\delta_0,0,\delta_0,0)$ is bounded by some constant $L$, $\mathbb P$-a.s.
   \item[(2)] There exist some constants $\lambda> 0$ and $\alpha\in[0,1)$ such that, $\mathbb P$-a.s., for all $t\in[0,T]$, $y,\bar{y},k,\bar{k}\in \mathbb{R}$, $\mu,\bar{\mu}\in\mathcal{P}_2(\mathbb{R})$ and $z,\bar{z}\in\mathbb{R}^d$, $\nu,\bar{\nu}\in\mathcal{P}_2(\mathbb{R}^d)$,
	\begin{align*}
		| f(t,y,\mu,z,\nu,k) - f(t,\bar{y},\bar{\mu},\bar{z},\bar{\nu},\bar{k})| &\leq \lambda( |y-\bar{y}| +(1+|z|+|\bar{z}|)|z-\bar{z}|+|k-\bar{k}|)\\
    &\quad +\lambda \mathcal{W}_2(\mu,\bar{\mu})+\lambda(1+(\mathcal{W}_2(\nu,\delta_0))^{\alpha}+(\mathcal{W}_2(\bar{\nu},\delta_0))^{\alpha})\mathcal{W}_2(\nu,\bar{\nu}).
	\end{align*}
    \end{enumerate}
\item[($H_G$)] For all $t\in[0,T]$, $G_t:\mathcal{C}_T(\mathbb{R})\rightarrow \mathbb{R}$ satisfies that for $y,\bar{y}\in\mathcal{C}_T(\mathbb{R})$, $G_t(0)=0$ and
    \begin{align*}
    G_t(y)=G_t(\{y_{s\wedge t}\}_{0\leq s\leq t\leq T}),~\text{ and }~|G_t(y)-G_t(\bar{y})|\leq\sup_{0\leq u\leq t}|y_u-\bar{y}_u|.
    \end{align*}
\item[($H_\ell$)] The running loss function $\ell : \Omega\times[0,T]\times\mathbb{R} \rightarrow \mathbb{R}$ is an $\mathscr{F}_T\times\mathcal{ B}(\mathbb{R})\times \mathcal{ B}(\mathbb{R})$-measurable map and there exists some constant $C>0$ such that, $\mathbb P$-a.s.,
\begin{enumerate}
	\item $(t,y)\rightarrow \ell(t,y)$ is continuous,
	\item $\forall t\in[0,T]$, $y\rightarrow\ell(t,y)$ is strictly increasing,
	\item  $\forall t\in[0,T]$, $\mathbb{E}[\ell(t,\infty)]>0$,
	\item $\forall t\in[0,T]$, $\forall y\in\mathbb{R} $, $|\ell(t,y)| \leq C(1+|y|)$.
\end{enumerate}
\end{itemize}
As in \cite{BH,HHLLW}, we introduce the operator $L_t$ which is crucial to deal with mean reflection. $L_t:  \mathcal{L}^2 \rightarrow [0,\infty) $, $t\in[0,T]$ is given by
\begin{eqnarray*}
	L_t :X \rightarrow \inf\{ x\geq 0 : \mathbb{E}[\ell(t,x+X)] \geq 0 \},
\end{eqnarray*}
 which is well-defined due to the Assumption $(H_{\ell})$ (see \cite{BH}).

 We make the following assumption for the operator $L_t$ and recall a sufficient condition from \cite{BH} such that the following assumption holds true in Remark \ref{remark}.
 \begin{itemize}
 \item[($H_L$)] There exists a constant $C>0$ such that for each $t\in[0,T]$,
 \begin{equation*}
|L_t(X)-L_t(Y)|\leq C\mathbb{E}[|X-Y|], \ \forall  X,Y \in \mathcal{L}^2.
\end{equation*}
\end{itemize}
\begin{remark}\label{remark}{\upshape
Assume that $(H_{\ell})$ holds true. Suppose that $\ell$ is a
bi-Lipschitz function in $x$, i.e., there exist some constants $0<c_{\ell}\leq C_{\ell}$ such that, $\mathbb P$-a.s., for all $t\in[0,T]$ and for all $x,y\in\mathbb{R}$,  \begin{equation*}
		{c_{\ell}} |x-y| \leq |\ell(t,x)-\ell(t,y) | \leq C_{\ell} |x-y|,
	\end{equation*}
Then $(H_L)$ holds true with $C=\frac{C_{\ell} }{c_{\ell} }$ (see \cite{BH}).
}
\end{remark}
\begin{remark}
  Assumption $(H^q_f)$ can be satisfied easily. For instance, the following generator satisfies assumption $(H^q_f)$: for all $t\in[0,T],y\in\mathbb{R}$, $\mu\in\mathcal{P}_2(\mathbb{R})$, $z\in\mathbb{R}^d$, $\nu\in\mathcal{P}_2(\mathbb{R}^d)$, $k\in\mathbb{R}$,
  \begin{equation*}
    f(t,y,\mu,z,\nu,k)=1+|y|+|z|+|z|^2+\mathcal{W}_2(\mu,\delta_0)+\left(\mathcal{W}_2(\nu,\delta_0)\right)^{\frac{3}{2}}.
  \end{equation*}
\end{remark}
\begin{remark}
If $\alpha=1$ in $(H^q_f)$, our results still hold by additionally assuming that the local Lipschitz constant of $f$ with respect to $\nu$ is small enough.  
\end{remark}
\subsection{Lipschitz case}
In this section, we study the existence and uniqueness theorem of equation \eqref{my1} with Lipschitz generator and square integrable terminal condition. As in \cite{BH}, we introduce the following definition of deterministic flat solutions of mean field BSDEs with mean reflection and nonlinear resistance in the Lipschitz case.
\begin{definition}
A triple of processes $(Y,Z,K)\in\mathcal{S}^{2}\times \mathcal{H}^2\times \mathcal{A}^2_D$
is said to be a deterministic solution to the mean-field BSDE \eqref{my1} with mean reflection and nonlinear resistance if it ensures that the equation \eqref{my1} holds true. A solution is said to be ``{\it flat}'' if moreover that $K$ increases only when needed, i.e., when we have
\begin{equation*}
		\int_0^T \mathbb{E}[\ell(t,Y_t)]  dK_t = 0.
\end{equation*}
\end{definition}

The first main result of this section is on the existence and uniqueness of the local solution for the mean field BSDE with mean reflection and nonlinear resistance, which reads as follows:
\begin{theorem}\label{Thm1L}
Assume that $(H_{\xi})-(H_f)-(H_G)-(H_{\ell})-(H_L)$ hold. Then there exists a constant $\varepsilon>0$ depending only on $C, L$ and $\lambda$ such that for any $T\in (0,\varepsilon]$, the mean-field BSDE \eqref{my1} with mean reflection and nonlinear resistance admits a unique deterministic flat solution $(Y,Z,K)\in \mathcal{S}^{2}\times \mathcal{H}^2\times\mathcal{A}^2_D$.
\end{theorem}

In fact, under assumptions $(H_{\xi})-(H_f)-(H_G)-(H_{\ell})-(H_L)$, for each $(y,z,k)\in \mathcal{S}^{2}\times\mathcal{H}^2\times\mathcal{A}^2_D$, it follows from \cite[Theorem 9]{BH} that the following BSDE with mean reflection
\begin{align}\label{1L}
\begin{cases}
&Y_t^{y,z,k}=\xi+\int^T_t f(s,Y^{y,z,k}_s,\mathbb{P}_{y_s},Z^{y,z,k}_s,\mathbb{P}_{z_s},G_s(k))ds-\int^T_t Z_s^{y,z,k} dB_s+K_T^{y,z,k}-K_t^{y,z,k};\\
& \ \ \ \ \ \ \ \ \ \ \ \ \ \ \ \ \ \ \ \ \ \ \ \mathbb{E}[\ell(t,Y_t^{y,z,k})]\geq 0,
\end{cases} \  \  t\in [0, T],
\end{align}
has a unique  deterministic flat solution $(Y^{y,z,k},Z^{y,z,k},K^{y,z,k})\in \mathcal{S}^{2}\times\mathcal{H}^2\times\mathcal{A}^2_D$.
\begin{lemma} \label{Lem1L}
Assume that $(H_{\xi})-(H_f)-(H_G)-(H_{\ell})-(H_L)$ hold. For $i=1,2$, let 
\begin{equation*}
  (Y^i,Z^i,K^i):=(Y^{y^i,z^i,k^i}, Z^{y^i,z^i,k^i},K^{y^i,z^i,k^i})
\end{equation*}
be the solution to the BSDE \eqref{1L} with mean reflection associated with the data $(y^i,z^i,k^i)$, then there exists constant $\varepsilon>0$ depending only on $C, L$ and $\lambda$ such that for any $T\in (0,\varepsilon]$, we have
\begin{align*}
&\|Y^{1}-Y^{2}\|^2_{\mathcal{S}^{2}}+\|Z^{1}-Z^{2}\|^2_{\mathcal{H}^2}+\sup_{0\leq t\leq T}|K^{1}_t-K^{2}_t|^2\\
&\leq \frac{1}{2}\left(\|y^{1}-y^{2}\|^2_{\mathcal{S}^{2}}+\|z^{1}-z^{2}\|^2_{\mathcal{H}^2}+\sup_{0\leq t\leq T} |k^1_t-k^2_t|^2\right).
\end{align*}
\end{lemma}
\begin{proof}
In order to simplify notations, for each $i=1,2$, we will denote
\begin{align*}
f^i_{\cdot}:=f(\cdot,Y^{i}_{\cdot},\mathbb{P}_{y^i_{\cdot}},Z_{\cdot}^{i},\mathbb{P}_{z^i_{\cdot}},G_{\cdot}(k^i)).
\end{align*}
Since for each $i=1,2$ and $t\in[0,T]$,
\begin{equation*}
	K^{i}_T-K^{i}_t= \sup_{t\leq s\leq T} L_s(X_s^{i})
\end{equation*}
with $X_t^{i} := \mathbb{E}_t\left[\xi+\int^T_t f^i_sds\right]$, we have
 \begin{align}
&\sup_{0\leq t\leq T} |(K^{1}_T-K^{1}_t)-(K^{2}_T-K^{2}_t)|\notag\\
&\leq  \sup_{0\leq s\leq T}|L_s(X_s^{1})-L_s(X_s^{2})|\label{2L}\\
 &\leq C\lambda\mathbb{E}\left[\int^T_{0} (|Y^{^1}_s-Y^{2}_s|+\mathbb{E}[|y^1_s-y^2_s|^2]^{\frac{1}{2}}+|G_s(k^1)-G_s(k^2)|)dr\right]\notag\\
 &\quad+C\lambda\mathbb{E}\left[\int_0^T(|Z_s^{^1}-Z_s^{2}|+\mathbb{E}[|z^1_s-z^2_s|^2]^{\frac{1}{2}})dr\right]\notag\\
 &\leq C\lambda T(\|Y^{1}-Y^{2}\|_{\mathcal{S}^{2}}+\|y^1-y^2\|_{\mathcal{S}^2}+\sup_{0\leq t\leq T} |k^1_t-k^2_t|)\notag\\
 &\quad+C\lambda\sqrt{T}(\|Z^{1}-Z^{2}\|_{\mathcal{H}^2}+\sqrt{T}\|z^1-z^2\|_{\mathcal{H}^2}).
   \nonumber\end{align}
Note that for each $t\in[0,T]$,
\begin{align*}
Y^{1}_t-Y^{2}_t+\int^T_t(Z^{1}_s-Z^{2}_s)dB_s=\int^T_t(f^1_s-f^2_s)ds+(K^{1}_T-K^{1}_t)-(K^{2}_T-K^{2}_t).
\end{align*}
Taking conditional expectation with respect to $\mathcal{F}_t$, we have
\begin{align*}
Y^{1}_t-Y^{2}_t= \mathbb{E}_t\left[\int^T_t(f^1_s-f^2_s)ds\right]+(K^{1}_T-K^{1}_t)-(K^{2}_T-K^{2}_t).
\end{align*}
Using Burkholder-Davis-Gundy inequality, we have
\begin{align*}
&\mathbb{E}\left[\sup_{0\leq t\leq T}|Y^{1}_t-Y^{2}_t|^2\right]\\
&\leq 8\mathbb{E}\left[\left(\int^T_0|f^1_s-f^2_s|ds\right)^2\right]+2\sup_{0\leq t\leq T}|(K^{1}_T-K^{1}_t)-(K^{2}_T-K^{2}_t)|^2\\
&\leq 8T\mathbb{E}\left[\int^T_0|f^1_s-f^2_s|^2ds\right]+2\sup_{0\leq t\leq T}|(K^{1}_T-K^{1}_t)-(K^{2}_T-K^{2}_t)|^2\\
&\leq 40\lambda^2T\mathbb{E}\left[\int^T_{0} (|Y^{1}_s-Y^{2}_s|^2+\mathbb{E}[|y^1_s-y^2_s|^2]+|Z_s^{1}-Z_s^{2}|^2+\mathbb{E}[|z^1_s-z^2_s|^2]+|G_s(k^1)-G_s(k^2)|^2)dr\right]\\
&\quad+2\sup_{0\leq t\leq T}|(K^{1}_T-K^{1}_t)-(K^{2}_T-K^{2}_t)|^2.
\end{align*}
Therefore, it holds that
\begin{align*}
\|Y^{1}-Y^{2}\|_{\mathcal{S}^{2}}&\leq 10(4+C^2)\lambda^2T^2\left(\|Y^{1}-Y^{2}\|^2_{\mathcal{S}^{2}}+\|y^{1}-y^{2}\|^2_{\mathcal{S}^2}+\sup_{0\leq t\leq T} |k^1_t-k^2_t|^2\right)\\
&\quad +10(4+C^2)\lambda^2T\left(\|Z^{1}-Z^{2}\|^2_{\mathcal{H}^{2}}+\|z^{1}-z^{2}\|^2_{\mathcal{H}^2}\right).
\end{align*}
On the other hand, we have
\begin{align*}
&Y^{1}_0-Y^{2}_0+\int^T_0(Z^{1}_s-Z^{2}_s)dB_s=\int^T_0(f^1_s-f^2_s)ds+(K^{1}_T-K^{1}_0)-(K^{2}_T-K^{2}_0).
\end{align*}
Taking square and taking expectation on both sides imply that
\begin{align*}
&\mathbb{E}\left[\int_0^T|Z^{1}_s-Z^{2}_s|^2ds\right]\\
&\leq 2\mathbb{E}\left[\left(\int^T_0|f^1_s-f^2_s|ds\right)^2\right]+2\sup_{0\leq t\leq T}|(K^{1}_T-K^{1}_t)-(K^{2}_T-K^{2}_t)|^2\\
&\leq 10(1+C^2)\lambda^2T^2\left(\|Y^{1}-Y^{2}\|^2_{\mathcal{S}^{2}}+\|y^{1}-y^{2}\|^2_{\mathcal{S}^2}+\sup_{0\leq t\leq T} |k^1_t-k^2_t|^2\right)\\
&\quad +10(1+C^2)\lambda^2T\left(\|Z^{1}-Z^{2}\|^2_{\mathcal{H}^{2}}+\|z^{1}-z^{2}\|^2_{\mathcal{H}^2}\right).
\end{align*}
Meanwhile, it holds that
\begin{equation*}
K^{1}_t=Y^{1}_0-Y^{1}_t-\int_0^tf^1_sds+\int_0^tZ^{1}_sdB_s
\end{equation*}
and
\begin{equation*}
K^{2}_t=Y^{2}_0-Y^{2}_t-\int_0^tf^2_sds+\int_0^tZ^{2}_sdB_s.
\end{equation*}
Therefore, we have
\begin{align*}
\sup_{0\leq t\leq T}|K^{k^1}_t-K^{k^2}_t|\leq |Y^{1}_0-Y^{2}_0|+\sup_{0\leq t\leq T}E\left[|Y^{1}_t-Y^{2}_t|\right]+\sup_{0\leq t\leq T}E\left[\bigg|\int_0^tf^1_sds-\int_0^tf^2_sds\bigg|\right].
\end{align*}
Hence, we obtain
\begin{align*}
&\sup_{0\leq t\leq T}|K^{1}_t-K^{2}_t|^2\\
&\leq 8\mathbb{E}\left[\sup_{0\leq t\leq T}|Y^{1}_t-Y^{2}_t|^2\right]+2\mathbb{E}\left[\left(\int_0^T|f^1_s-f^2_s|ds\right)^2\right]\\
&\leq 10(33+8C^2)\lambda^2T^2\left(\|Y^{1}-Y^{2}\|^2_{\mathcal{S}^{2}}+\|y^{1}-y^{2}\|^2_{\mathcal{S}^2}+\sup_{0\leq t\leq T} |k^1_t-k^2_t|^2\right)\\
&\quad +10(33+8C^2)\lambda^2T\left(\|Z^{1}-Z^{2}\|^2_{\mathcal{H}^{2}}+\|z^{1}-z^{2}\|^2_{\mathcal{H}^2}\right).
\end{align*}
Therefore, we have
\begin{align*}
&\|Y^{1}-Y^{2}\|^2_{\mathcal{S}^{2}}+\|Z^{1}-Z^{2}\|^2_{\mathcal{H}^2}+\sup_{0\leq t\leq T}|K^{1}_t-K^{2}_t|^2\\
&\leq 10(38+10C^2)\lambda^2T^2\left(\|Y^{1}-Y^{2}\|^2_{\mathcal{S}^{2}}+\|y^{1}-y^{2}\|^2_{\mathcal{S}^2}+\sup_{0\leq t\leq T} |k^1_t-k^2_t|^2\right)\\
&\quad +10(38+10C^2)\lambda^2T\left(\|Z^{1}-Z^{2}\|^2_{\mathcal{H}^{2}}+\|z^{1}-z^{2}\|^2_{\mathcal{H}^2}\right).
\end{align*}
Now we define
\begin{equation}\label{3L}
\varepsilon:=\min\left( \sqrt{\frac{1}{40(38+10C^2)\lambda^2 }}, \frac{1}{40(38+10C^2)\lambda^2 }\right),
\end{equation}
and it is straightforward to check that for any $T \in (0, \varepsilon]$,
\begin{align*}
&\frac{3}{4}\|Y^{1}-Y^{2}\|^2_{\mathcal{S}^{2}}+\frac{3}{4}\|Z^{1}-Z^{2}\|^2_{\mathcal{H}^2}+\sup_{0\leq t\leq T}|K^{1}_t-K^{2}_t|^2\\
&\leq \frac{1}{4}\left(\|y^{1}-y^{2}\|^2_{\mathcal{S}^{2}}+\|z^{1}-z^{2}\|^2_{\mathcal{H}^2}+\sup_{0\leq t\leq T} |k^1_t-k^2_t|^2\right).
\end{align*}
which completes the proof.
\end{proof}

Then we give the proof of Theorem \ref{Thm1L}.

\medskip
\begin{proof}[Proof of Theorem \ref{Thm1L}]
We take $\varepsilon$ as \eqref{3L}. For $T\in (0,\varepsilon]$, define $(Y^0,Z^0,K^0)=(0,0,0)$ and by \eqref{1L}, define $(Y^1, Z^1, K^1):=(Y^{Y^0,Z^0,K^0}, Z^{Y^0,Z^0,K^0}, K^{Y^0,Z^0,K^0})$. By recurrence, for each $i\in \mathbb{N}$, set
\begin{equation}\label{4L}
(Y^{i+1}, Z^{i+1}, K^{i+1}):=(Y^{Y^i,Z^i,K^i}, Z^{Y^i,Z^i,K^i}, K^{Y^i,Z^i,K^i}).
\end{equation}
It follows from Lemma \ref{Lem1L} that there exist $Y\in \mathcal{S}^2$, $Z\in \mathcal{H}^2$ and $K\in \mathcal{A}^2_D$ such that
\begin{equation}\label{5L}
\sup_{0\leq t\leq T}|{K}^{n}_t-{K}_t|^2\longrightarrow 0,\quad \|{Y}^{n}-{Y}\|_{\mathcal{S}^{2}}\longrightarrow 0 \qquad {\rm and}\qquad
\|{Z}^{n}-{Z}\|_{\mathcal{H}^2}\longrightarrow 0.
\end{equation}
By a standard argument, we have for each $t\in [0, T]$, in $\mathcal{L}^2(\mathscr{F}_T)$,
\begin{align*}
&\int^T_t f(s, Y^n_s,\mathbb{P}_{Y^{n-1}_s}, Z^{n}_s,\mathbb{P}_{Z^{n-1}_s},G_s(K^{n-1}))ds\\
&\longrightarrow \int^T_t f(s, Y_s,\mathbb{P}_{Y_s}, Z_s,\mathbb{P}_{Z_s}, G_s(K))ds.
\end{align*}
Thus, the triple $(Y, Z, K)$ is a solution to the mean-field BSDE (\ref{my1}) with mean reflection and nonlinear resistance and we only need to prove that the solution $(Y, Z, K)$ is ``flat''. Indeed, it is easy to check that
$$
K_T-K_t =\lim_{n\rightarrow \infty} K^n_T-K^n_t
=\lim_{n\rightarrow \infty}\sup_{t\leq s\leq T} L_s(X^n_s),$$
where
\begin{equation*}
X_t^n := \mathbb{E}_t\left[\xi+\int^T_t f(s,Y^{n}_s,\mathbb{P}_{Y^{n-1}_s},Z^n_s,\mathbb{P}_{Z^{n-1}_s},G_s(K^{n-1}))ds\right].
\end{equation*}
By \eqref{4L} and \eqref{5L}, we deduce
$$\|X^n-X\|_{{\mathcal{S}^{2}}}\longrightarrow 0,$$ where
\begin{equation*}
X_t := \mathbb{E}_t\left[\xi+\int^T_t f(s,Y_s,\mathbb{P}_{Y_s},Z_s,\mathbb{P}_{Z_s},G_s(K))ds\right].
\end{equation*}
Therefore, $K_T-K_t =
\sup_{t\leq s\leq T} L_s(X_s)$, which implies the ``flatness''. The uniqueness follows immediately from Lemma \ref{Lem1L}. The proof is complete.

\end{proof}
\medskip

We now construct a global solution to equation \eqref{my1} for arbitrarily large time horizon.
\begin{theorem}\label{Thm2L}
Assume that $(H_{\xi})-(H_f)-(H_G)-(H_{\ell})-(H_L)$ hold. Then for arbitrarily large $T$, the mean-field BSDE \eqref{my1} with mean reflection admits a unique deterministic flat solution $(Y,Z,K)\in \mathcal{S}^{2}\times \mathcal{H}^2\times\mathcal{A}^2_D$ on $[0,T]$.
\end{theorem}

\medskip

\begin{proof} The proof is divided into two steps.

{\it Step 1.} For given $k\in\mathcal{A}^2_{D}$, we will show that the following BSDE with mean reflection
\begin{align}\label{my2}
  \begin{cases}
  &Y_t=\xi+\int^T_t f(s,Y_s,\mathbb{P}_{Y_s},Z_s,\mathbb{P}_{Z_s},G_s(k))ds-\int^T_t Z_s dB_s+K_T-K_t;\\
  & \ \ \ \ \ \ \ \ \ \ \ \ \ \ \ \ \ \ \mathbb{E}[\ell(t,Y_t)]\geq 0,
  \end{cases} \  \  \forall t\in[0,T],
  \end{align}
  admits a unique deterministic flat solution $(Y,Z,K)\in \mathcal{S}^{2}\times \mathcal{H}^2\times\mathcal{A}^2_D$. In fact, by Theorem \ref{Thm1L}, there exists some constant $\varepsilon>0$ depending only on $L, \lambda$ and $C$ such that the mean-field BSDE \eqref{my2} with mean reflection admits a unique deterministic flat  solution$(Y^1,Z^1,K^1)\in \mathcal{S}^{2}_{[T-\varepsilon,T]}\times\mathcal{H}^2_{[T-\varepsilon,T]}\times\mathcal{A}^2_{D,[T-\varepsilon,T]}$ on the time interval $[T-\varepsilon,T]$. Next we take  $T -\varepsilon$ as the terminal time and
apply Theorem \ref{Thm1L} again to find the unique deterministic flat solution of the mean-field BSDE \eqref{my2} with mean reflection $(Y^2,Z^2,K^2)\in \mathcal{S}^{2}_{[T-2\varepsilon,T-\varepsilon]}\times\mathcal{H}^2_{[T-2\varepsilon,T-\varepsilon]}\times\mathcal{A}^2_{D,[T-2\varepsilon,T-\varepsilon]}$ on the time interval $[T-2\varepsilon,T-\varepsilon]$. Let us set
\[
{Y}_t=\sum\limits_{i=1}^2Y^i_t\mathbf{1}_{[T-i\varepsilon,T-(i-1)\varepsilon)}+Y^1_T\mathbf{1}_{\{T\}}, \  {Z}_t=\sum\limits_{i=1}^2Z^i_t\mathbf{1}_{[T-i\varepsilon,T-(i-1)\varepsilon)}+Z^1_T\mathbf{1}_{\{T\}}
\]
on $[T-2\varepsilon,T]$ and ${K}_t=K^2_t$ on $[T-2\varepsilon,T-\varepsilon)$, ${K}_t=K^2_{T-\varepsilon}+K^1_t$ on $[T-\varepsilon,T]$. One can easily check that $(Y,Z,K)\in\mathcal{S}^{2}_{[T-2\varepsilon,T]}\times\mathcal{H}^2_{[T-2\varepsilon,T]}\times\mathcal{A}^2_{D,[T-2\varepsilon,T]}$ is a deterministic flat solution to the mean-field BSDE \eqref{my2} with mean reflection on $[T-2\varepsilon,T]$.

Furthermore, we repeat this procedure so that we can build a deterministic flat
solution  $(Y,Z,K)\in\mathcal{S}^{2}\times\mathcal{H}^2\times\mathcal{A}^2_{D}$ to the mean-field BSDE \eqref{my2} with mean reflection on $[0,T]$. The uniqueness of the global solution on $[0, T]$ is inherited from the uniqueness of local solution on each time interval.

{\it Step 2.} Let $N$ be the smallest integer such that $N\geq \frac{T}{\varepsilon}$. For $k^i\in\mathcal{A}^2_D$, $i=1,2$, let $(Y^i,Z^i,K^i)$ be the unique deterministic flat solution to \eqref{my2} with data $k^i$. It follows from a similar argument as in the proof of Lemma \ref{Lem1L} that for any $1\leq i\leq N$,
\begin{align*}
  \sup_{(i-1)\varepsilon\leq t\leq i\varepsilon \wedge T}|K^{1}_t-K^{2}_t|^2\leq \frac{1}{2}\left(\sup_{0\leq t\leq T} |k^1_t-k^2_t|^2\right).
  \end{align*}  
Since $K^1,K^2$ are determinsitic, we therefore have 
\begin{align*}
  \sup_{0\leq t\leq T}|K^{1}_t-K^{2}_t|^2\leq \max_{1\leq i\leq N}\left(\sup_{(i-1)\varepsilon\leq t\leq i\varepsilon \wedge T}|K^{1}_t-K^{2}_t|^2\right)\leq \frac{1}{2}\left(\sup_{0\leq t\leq T} |k^1_t-k^2_t|^2\right).
\end{align*}
Thus it follows from stanadard contraction mapping argument that mean-field BSDE \eqref{my1} with mean reflection admits a unique deterministic flat solution $(Y,Z,K)\in \mathcal{S}^{2}\times \mathcal{H}^2\times\mathcal{A}^2_D$ on $[0,T]$.
\end{proof}

\subsection{Quadratic case}

In this section, we study the existence and uniqueness theorem of equation \eqref{my1} with quadratic generator and bounded terminal condition. As in \cite{BH}, we introduce the following definition of deterministic flat solutions of mean-field BSDEs with mean reflection and nonlinear resistance in the quadratic case (see also \cite{HHLLW}).
\begin{definition}
A triple of processes $(Y,Z,K)\in\mathcal{S}^{\infty}\times BMO\times \mathcal{A}^{\infty}_D$
is said to be a deterministic solution to the mean-field BSDE \eqref{my1} with mean reflection and nonlinear resistance if it ensures that the equation \eqref{my1} holds true. A solution is said to be ``{\it flat}'' if moreover that
 $K$ increases only when needed, i.e., when we have
\begin{equation*}
		\int_0^T \mathbb{E}[\ell(t,Y_t)]  dK_t = 0.
\end{equation*}
\end{definition}

 Since $L_t(0)$ is continuous in $t$,  without loss of generality we assume that $|L_t(0)|\leq L$ for each $t\in[0,T]$, see \cite{BH}. The first main result of this section is on the existence and uniqueness of the local solution for the mean-field BSDE with mean reflection and nonlinear resistance, which reads as follows:
\begin{theorem}\label{Thm1q}
Assume that $(H^q_{\xi})-(H^q_f)-(H_G)-(H_{\ell})-(H_L)$ hold. Then, there exists a sufficiently large constant $\tilde{A}>0$ and a constant $\widehat{\delta}^{\tilde{A}}>0$ depending only on $\tilde{A}, C, L$ and $\lambda$, such that for any $T\in (0,\widehat{\delta}^{\tilde{A}}]$, the mean-field BSDE \eqref{my1} with mean reflection and nonlinear resistance admits a unique
 deterministic flat solution $(Y,Z,K)\in \mathcal{S}^{\infty}\times BMO\times\mathcal{A}^{\infty}_D$ such that
\[
\|Y\|_{\mathcal{S}^{\infty}}\leq \tilde{A},~~\|Z\|_{BMO}\leq\tilde{A},~\text{and}~\sup_{0\leq t\leq T}|K_t|\leq \tilde{A}.
\]
\end{theorem}

Indeed, under assumptions $(H^q_{\xi})-(H^q_f)-(H_G)-(H_{\ell})-(H_L)$, for each  $(y,z,k)\in \mathcal{S}^{\infty}\times BMO\times\mathcal{A}^{\infty}_D$,  it follows from \cite[Theorem 2.4]{HHLLW} that the following quadratic BSDE with mean reflection
\begin{align}\label{1q}
\begin{cases}
&Y_t^{y,z,k}=\xi+\int^T_t f(s,y_s,\mathbb{P}_{y_s},Z_s^{y,z,k},\mathbb{P}_{z_s},G_s(k))ds-\int^T_t Z_s^{y,z,k} dB_s+K_T^{y,z,k}-K_t^{y,z,k};\\
& \ \ \ \ \ \ \ \ \ \ \ \ \ \ \ \ \ \ \ \ \ \ \ \mathbb{E}[\ell(t,Y_t^{y,z,k})]\geq 0,
\end{cases} \  \  t\in [0, T],
\end{align}
has a unique  deterministic flat solution  $(Y^{y,z,k},Z^{y,z,k},K^{y,z,k})\in \mathcal{S}^{\infty}\times BMO\times\mathcal{A}^{\infty}_D$.  Then we define the solution map
$\Gamma: (y,z,k)\rightarrow\Gamma(y,z,k)$ by
\[
\Gamma(y,z,k):=(Y^{y,z,k},Z^{y,z,k},K^{y,z,k}), \ \forall (y,z,k)\in \mathcal{S}^{\infty}\times BMO\times\mathcal{A}^{\infty}_D.
\]

In order to show that $\Gamma$ is contractive, for each real number $\tilde{A}\geq \tilde{A}_0$ we consider the following set:
\begin{eqnarray*}
  \mathscr{B}_{\tilde{A}}:=\left\{ (y,z,k)\in \mathcal{S}^{\infty}\times BMO\times\mathcal{A}^{\infty}_D: \|y\|_{\mathcal{S}^{\infty}}\leq \tilde{A}, \|z\|_{BMO}\leq \tilde{A},  \sup_{0\leq t\leq T}|k_t|\leq \tilde{A} \right\},
\end{eqnarray*}
where \begin{equation}\label{2q}
\tilde{A}_0=(4+3C)L+\left(1+C\lambda\right)\left(1+\frac{3L}{\lambda}\right)e^{9\lambda L}.
\end{equation}

\begin{lemma}\label{Lem1q}
Assume that $(H^2_{\xi})-(H^q_f)-(H_G)-(H_{\ell})-(H_L)$ hold and $\tilde{A}\geq \tilde{A}_0$, where $\tilde{A}_0$ is defined by \eqref{2q}.  Then there is a constant $\delta^{\tilde{A}}>0$ depending only on $L, \lambda, C$ and $\tilde{A}$ such that for any $T\in (0,\delta^{\tilde{A}}]$,
$\Gamma(\mathscr{B}_{\tilde{A}})\subset \mathscr{B}_{\tilde{A}}$.
\end{lemma}
\begin{proof}
In order to simplify notations, we will denote
\begin{align*}
Y:=Y^{y,z,k},\quad Z:=Z^{y,z,k}, \quad K:=K^{y,z,k}.
\end{align*}
It is easy to verify that for each $t\in[0,T]$,
\[
(Y_t,Z_t)=(\overline{Y}_t+(K_T-K_t),\overline{Z}_t),
\]
where $(\overline{Y}, \overline{Z})\in \mathcal{S}^{\infty}\times BMO$ is the solution to the following standard BSDE on the time interval $[0,T]$
\begin{align}\label{3q}
\overline{Y}_t=\xi+\int^T_t f(s,y_s,\mathbb{P}_{y_s},\overline{Z}_s,\mathbb{P}_{z_s},G_s(k))ds-\int^T_t \overline{Z}_s dB_s,
\end{align}
and for each $t\in[0,T]$,
\begin{equation*}
	K_T-K_t= \sup_{t\leq s\leq T} L_s(X_s)
\end{equation*}
with
\begin{equation*}
 X_t := \mathbb{E}_t\left[\xi+\int^T_t f(s,y_{s},\mathbb{P}_{y_{s}},Z_{s},\mathbb{P}_{z_{s}},G_{s}(k))ds\right].
 \end{equation*}
 Consequently, we obtain that
 \begin{align}\label{4q}
 \|{Y}\|_{\mathcal{S}^{\infty}}\leq \|\overline{Y}\|_{\mathcal{S}^{\infty}}+ \sup_{0\leq s\leq T}L_s(X_s).
  \end{align}
The remainder of the proof will be in two steps.

{\it Step 1. The estimate of  $\overline{Y}$.}
Since $\overline{Z}\in BMO $, we can find a vector process $\beta\in BMO$  such that
\begin{eqnarray*}
f(s,y_s,\mathbb{P}_{y_s},\overline{Z}_s,\mathbb{P}_{z_s},G_s(k))-f(s,y_s,\mathbb{P}_{y_s},0,\mathbb{P}_{z_s},G_s(k))=\overline{Z}_s \beta_s, \ \forall s\in[0,T].
\end{eqnarray*}

Then $\widetilde {B}_t:=B_t-\int_{0}^t\beta_sds$, defines a Brownian motion under the equivalent probability measure $\widetilde{\mathbb{P}}$ given by
$$d\widetilde{\mathbb{P}}: =\mathscr{Exp} (\beta\cdot B)_{T} d\mathbb{P}.$$
Thus by the equation \eqref{3q}, we have
\begin{align*}
\overline{Y}_t&=\mathbb{E}^{\widetilde{\mathbb{P}}}_t\left[\xi+\int^T_t f(s,y_s,\mathbb{P}_{y_s},0,\mathbb{P}_{z_s},G_s(k))ds\right]\\
&\leq \mathbb{E}^{\widetilde{\mathbb{P}}}_t\left[\xi+\int^T_t \left(|f(s,0,0,0,0)|+\lambda\left(|y_s|+\mathbb{E}[|y_s|^2]^{\frac{1}{2}}+\mathbb{E}[|z_s|^2]^{\frac{1}{2}}+\mathbb{E}[|z_s|^2]^{\frac{1+\alpha}{2}}+|G_s(k)|\right)\right)ds\right]\\
&\leq (1+T)L+\lambda T\left(2\|y\|_{\mathcal{S}^{\infty}}+\sup_{0\leq t\leq T}|k_t|\right)+\lambda\int_t^T\mathbb{E}[|z_s|^2]^{\frac{1}{2}}ds+\lambda\int_t^T\mathbb{E}[|z_s|^2]^{\frac{1+\alpha}{2}}ds\\
&\leq (1+T)L+\lambda T\left(2\|y\|_{\mathcal{S}^{\infty}}+\sup_{0\leq t\leq T}|k_t|\right)+\lambda\sqrt{T}\left(\int_t^T\mathbb{E}[|z_s|^2]ds\right)^{\frac{1}{2}}+\lambda T^{\frac{1-\alpha}{2}}\left(\int_t^T\mathbb{E}[|z_s|^2]ds\right)^{\frac{1+\alpha}{2}}
\end{align*}
which implies that
\begin{align}\label{5q}
\|\overline{Y}\|_{\mathcal{S}^{\infty}}&\leq (1+T)L+\lambda T\left(2\|y\|_{\mathcal{S}^{\infty}}+\sup_{0\leq t\leq T}|k_t|\right)+\lambda\sqrt{T}\|z\|_{BMO}+\lambda T^{\frac{1-\alpha}{2}}\|z\|_{BMO}^{1+\alpha}\notag\\
&\leq (1+T)L+3\lambda T\tilde{A}+\lambda\sqrt{T}\tilde{A}+\lambda T^{\frac{1-\alpha}{2}}\tilde{A}^{1+\alpha}.
\end{align}
 Thus it follows from a similar technique as in \cite[Proposition 2.1]{BH1} that
\begin{align}\label{6q}
\|Z\|_{BMO}^2&\leq \frac{2e^{3\lambda\|\overline{Y}\|_{\mathcal{S}^{\infty}}}}{3}\left(\frac{(1+T)L}{\lambda}+ T\left(2\|y\|_{\mathcal{S}^{\infty}}+\sup_{0\leq t\leq T}|k_t|+\frac{1}{2}\right)+\sqrt{T}\|z\|_{BMO}+ T^{\frac{1-\alpha}{2}}\|z\|_{BMO}^{1+\alpha}\right)\notag\\
&\leq \frac{2e^{3\lambda\|\overline{Y}\|_{\mathcal{S}^{\infty}}}}{3}\left(\frac{(1+T)L}{\lambda}+ 3T\tilde{A}+\frac{T}{2}+\sqrt{T}\tilde{A}+ T^{\frac{1-\alpha}{2}}\tilde{A}^{1+\alpha}\right).
\end{align}
{\it Step 2. The estimate of $Y$.}
Thanks to the Assumption $(H_L)$, for each $s\in[0,T]$ we have\[
|L_s(X_s)-L_s(0)|\leq C\mathbb{E}[|X_s|].
\]
 Therefore from the definition $X$ and Assumption ($H^q_f$) we deduce that
 \begin{align}\label{7q}
 \sup_{0\leq s\leq T}L_s(X_s)\leq &L+ C\sup_{0\leq s\leq T}\mathbb{E}\left[|\xi|+\int^T_s |f(r,y_r,\mathbb{P}_{y_r},Z_r,\mathbb{P}_{z_r},G_r(k))|dr\right]\nonumber\\
 &\leq L+ C\mathbb{E}\left[|\xi|+\int^T_0 \left(|f(r,y_r,\mathbb{P}_{y_r},0,\mathbb{P}_{z_r},G_r(k))|+\frac{1}{2}\lambda+\frac{3}{2}\lambda|Z_r|^2\right)dr\right]\nonumber\\
 \leq & (1+C+CT)L+C\lambda T\left(2\|y\|_{\mathcal{S}^{\infty}}+\sup_{0\leq t\leq T}|k_t|+\frac{1}{2}\right)\notag\\
 &\quad+C\lambda\sqrt{T}\|z\|_{BMO}+C\lambda T^{\frac{1-\alpha}{2}}\|z\|_{BMO}^{1+\alpha}+\frac{3}{2}C\lambda\|Z\|_{BMO}^2\notag\\
  \leq & (1+C+CT)L+3C\lambda T\tilde{A}+\frac{C\lambda T}{2}+C\lambda\sqrt{T}\tilde{A}+C\lambda T^{\frac{1-\alpha}{2}}\tilde{A}^{1+\alpha}+\frac{3}{2}C\lambda\|Z\|_{BMO}^2.
 \end{align}
Then we define
\begin{equation}\label{8q}
\delta^{\tilde{A}}:=\min\left(\frac{L}{9\lambda \tilde{A}},\frac{L^2}{9\lambda^2 \tilde{A}^2},\left(\frac{L}{3\lambda \tilde{A}^{1+\alpha}}\right)^{\frac{2}{1-\alpha}}\right).
\end{equation}
Recalling equations \eqref{4q}, \eqref{5q}, \eqref{6q} and \eqref{7q}, we derive that for each $T\in(0, \delta^{\tilde{A}}]$,
\begin{align*}
&\|\overline{Y}\|_{\mathcal{S}^{\infty}}\leq 3L,\quad \|Z\|_{BMO}\leq\sqrt{\left(\frac{2L}{\lambda}+\frac{1}{3}\right)e^{9\lambda L}}\leq \tilde{A},\\
&\|K\|=K_T= \sup_{0\leq s\leq T}L_s(X_s)\leq (1+3C)L+C(\lambda+3L)e^{9\lambda L}\leq \tilde{A},\\
&\|Y\|_{\mathcal{S}^{\infty}}\leq \|\overline{Y}\|_{\mathcal{S}^{\infty}}+\sup_{0\leq s\leq T}L_s(X_s)\leq (4+3C)L+C(\lambda+3L)e^{9\lambda L}\leq \tilde{A}
\end{align*}
which is the desired result.
\end{proof}

Now we  show the contractive property of the solution map $\Gamma$.
\begin{lemma} \label{Lem2q}
Assume that $(H^q_{\xi})-(H^q_f)-(H_G)-(H_{\ell})-(H_L)$ hold and $\tilde{A}\geq \tilde{A}_0$, where $\tilde{A}_0$ is defined by \eqref{2q}. Then there exists a constant $\widehat{\delta}^{\tilde{A}}$ such that $0<\widehat{\delta}^{\tilde{A}}\leq\delta^{\tilde{A}}$ and for any $T\in (0,\widehat{\delta}^{\tilde{A}}]$, we have
\begin{align*}
&\|Y^{y^1,z^1,k^1}-Y^{y^2,z^2,k^2}\|_{\mathcal{S}^{\infty}}+\|Z^{y^1,z^1,k^1}-Z^{y^2,z^2,k^2}\|_{BMO}+\sup_{0\leq t\leq T}|K^{y^1,z^1,k^1}_t-K^{y^2,z^2,k^2}_t|\\
&\leq\frac{1}{2}\left(\|y^1-y^2\|_{\mathcal{S}^{\infty}}+\|z^1-z^2\|_{BMO}+\sup_{0\leq t\leq T}|k^1_t-k^2_t|\right),\quad \forall (y^1,z^1,k^1),(y^2,z^2,k^2)\in \mathscr{B}_{\tilde{A}}.
\end{align*}
\end{lemma}
\begin{proof}
In order to simplify notations, for each $i=1,2$, we will denote
\begin{align*}
&(Y^{i},Z^i,K^i)=\Gamma(y^i,z^i,k^i),\\
&f^i_{\cdot}=f({\cdot},y^{i}_{\cdot},\mathbb{P}_{y^i_{\cdot}},Z_{\cdot}^{i},\mathbb{P}_{z^i_{\cdot}},G_{\cdot}(k^i)).
\end{align*}
Again it is easy to verify that for each $t\in[0,T]$,
\begin{equation}\label{9q}
(Y^{i}_t, Z^{i}_t)=(\overline{Y}^{i}_t+(K^{i}_T-K^{i}_t),\overline{Z}^{i}_t),
\end{equation}
where $(\overline{Y}^{i}, \overline{Z}^{i})\in \mathcal{S}^{\infty}\times BMO $ is the solution to the following standard BSDE on the time interval $[0,T]$
\begin{align}\label{10q}
\overline{Y}^i_t=\xi+\int^T_t f(s,y^i_s,\mathbb{P}_{y^i_s},\overline{Z}^i_s,\mathbb{P}_{z^i_s},G_s(k^i))ds-\int^T_t \overline{Z}^i_s dB_s.
\end{align}
Since for each $t\in[0,T]$,
\begin{equation*}
	K^{i}_T-K^{i}_t= \sup_{t\leq s\leq T} L_s(X_s^{i})
\end{equation*}
with $X_t^{i}:=\mathbb{E}_t\left[\xi+\int^T_tf^i_sds\right]$, we have
 \begin{align}\label{11q}
&\sup_{0\leq t\leq T} |(K^{1}_T-K^{1}_t)-(K^{2}_T-K^{2}_t)|\notag\\
&\leq  \sup_{0\leq s\leq T}|L_s(X_s^{1})-L_s(X_s^{2})|\\
 &\leq C\lambda\mathbb{E}\left[\int^T_{0} (|y^{1}_r-y^{2}_r|+\mathbb{E}[|y^1_r-y^2_r|^2]^{\frac{1}{2}}+|G_r(k^1)-G_r(k^2)|)dr\right]\notag\\
 &\quad + C\lambda\mathbb{E}\left[\int^T_{0} ((1+|Z_r^{1}|+|Z_r^{2}|)|Z_r^{1}-Z_r^{2}|+(1+\mathbb{E}[|z^1_r|^2]^{\frac{\alpha}{2}}+\mathbb{E}[|z^2_r|^2]^{\frac{\alpha}{2}})\mathbb{E}[|z^1_r-z^2_r|^2]^{\frac{1}{2}})dr\right]\notag
\end{align}
 Applying H\"older's inequality, we obtain
 \begin{align}\label{12q}
  &\mathbb{E}\left[\int^T_{0}(1+|Z_r^{1}|+|Z_r^{2}|)|Z_r^{1}-Z_r^{2}|dr\right]\nonumber\\
  &\leq \sqrt{3} \mathbb{E}\left[\int^T_{0}(1+|Z_r^{1}|^2+|Z_r^{2}|^2)dr\right]^{\frac{1}{2}}\|Z^{1}-Z^{2}\|_{BMO}
\\&\leq \sqrt{3+6\tilde{A}^2}\|Z^{1}-Z^{2}\|_{BMO}.\notag
\end{align}
On the other hand, applying H\"older's inequality, it holds that
\begin{align}\label{13q}
 &\mathbb{E}\left[\int^T_{0} (1+\mathbb{E}[|z^1_r|^2]^{\frac{\alpha}{2}}+\mathbb{E}[|z^2_r|^2]^{\frac{\alpha}{2}})\mathbb{E}[|z^1_r-z^2_r|^2]^{\frac{1}{2}}dr\right]\notag\\
 &\leq \left[\int^T_{0} (1+\mathbb{E}[|z^1_r|^2]^{\frac{\alpha}{2}}+\mathbb{E}[|z^2_r|^2]^{\frac{\alpha}{2}})^2dr\right]^{\frac{1}{2}}\left[\int^T_{0}\mathbb{E}[|z^1_r-z^2_r|^2]dr\right]^{\frac{1}{2}}\notag\\
 &\leq \sqrt{3}\left[\int^T_{0} (1+\mathbb{E}[|z^1_r|^2]^{\alpha}+\mathbb{E}[|z^2_r|^2]^{\alpha})dr\right]^{\frac{1}{2}}\|z^1-z^2\|_{BMO}\\
 &\leq \sqrt{3}\left[T+T^{1-\alpha}\left(\int^T_{0} \mathbb{E}[|z^1_r|^2]\right)^{\alpha}dr+T^{1-\alpha}\left(\int^T_{0} \mathbb{E}[|z^2_r|^2]\right)^{\alpha}dr\right]^{\frac{1}{2}}\|z^1-z^2\|_{BMO}\notag\\
 &\leq \sqrt{3T+6 T^{1-\alpha}\tilde{A}^{2\alpha}}\|z^1-z^2\|_{BMO}.\notag
\end{align}
We recall the representation \eqref{9q} and conclude that
  \begin{align}\label{14q}
  \|{Y}^{1}-{Y}^{2}\|_{\mathcal{S}^{\infty}}&\leq  \|\overline{Y}^{1}-\overline{Y}^{2}\|_{\mathcal{S}^{\infty}}+ \sup_{0\leq t\leq T} |(K^{1}_T-K^{1}_t)-(K^{2}_T-K^{2}_t)|\notag\\
  &\leq  \|\overline{Y}^{1}-\overline{Y}^{2}\|_{\mathcal{S}^{\infty}}+ C\lambda T \left(2\|{y^{1}}-y^{2}\|_{\mathcal{S}^{\infty}}+\sup_{0\leq t\leq T}|k^1_t-k^2_t|\right)\\
  &\quad+C\lambda\left(\sqrt{3+6\tilde{A}^2}\|Z^{1}-Z^{2}\|_{BMO}+\sqrt{3T+6 T^{1-\alpha}\tilde{A}^{2\alpha}}\|z^1-z^2\|_{BMO}\right).
   \nonumber\end{align}
The remainder of the proof will be divided in two steps.

{\it Step 1. The estimate of $\|\overline{Y}^{1}-\overline{Y}^{2}\|_{\mathcal{S}^{\infty}}$.}
By the linearization argument, we can find a vector process $\widehat{\beta}\in BMO$ such that
\begin{align*}
f(s,y^{1}_s,\mathbb{P}_{y^1_s}, \overline{Z}^{1}_s,\mathbb{P}_{z^1_s},G_s(k^1))-f(s,y^{1}_s,\mathbb{P}_{y^1_s}, \overline{Z}^{2}_s,\mathbb{P}_{z^1_s},G_s(k^1))=  (\overline{Z}^{1}_s-\overline{Z}^{2}_s)\widehat{\beta}_s.
\end{align*}
Then $\widehat{B}_t:=B_t-\int_{0}^t\widehat{\beta}_sds$  defines a Brownian motion under the equivalent probability measure $\widehat{\mathbb{P}}$ given by
$$d\widehat{\mathbb{P}}: =\mathscr{Exp} (\widehat{\beta}\cdot B)_{T} d\mathbb{P}.$$
Thus by equation \eqref{10q}, we have
\begin{align*}
\overline{Y}^{1}_t-\overline{Y}^{2}_t&=\mathbb{E}^{\widehat{\mathbb{P}}}_t\left[\int^T_t (f(s,y^{1}_s,\mathbb{P}_{y^1_s},\overline{Z}^{2}_s,\mathbb{P}_{z^1_s},G_s(k^1))-f(s,y^{2}_s,\mathbb{P}_{y^2_s},\overline{Z}^{2}_s,\mathbb{P}_{z^2_s},G_s(k^2)))ds\right]\\
&\leq\lambda\mathbb{E}^{\widehat{\mathbb{P}}}_t\left[\int^T_t (|y^{1}_s-y^2_s|+\mathbb{E}[|y^1_s-y^2_s|^2]^{\frac{1}{2}}+(1+\mathbb{E}[|z^1_s|^2]^{\frac{\alpha}{2}}+\mathbb{E}[|z^2_s|^2]^{\frac{\alpha}{2}})\mathbb{E}[|z^1_s-z^2_s|^2]^{\frac{1}{2}}+|G_s(k^1)-G_s(k^2)|)ds\right]\\
&\leq \lambda T\left(2\|y^1-y^2\|_{\mathcal{S}^{\infty}}+\sup_{0\leq t\leq T}|k^1_t-k^2_t|\right)+\lambda\int_t^T(1+\mathbb{E}[|z^1_s|^2]^{\frac{\alpha}{2}}+\mathbb{E}[|z^2_s|^2]^{\frac{\alpha}{2}})\mathbb{E}[|z^1_s-z^2_s|^2]^{\frac{1}{2}}ds
\end{align*}
which implies that
\begin{align}\label{15q}
\|\overline{Y}^{1}-\overline{Y}^{2}\|_{\mathcal{S}^{\infty}}\leq \lambda T\left(2\|y^1-y^2\|_{\mathcal{S}^{\infty}}+\sup_{0\leq t\leq T}|k^1_t-k^2_t|\right)+\lambda \sqrt{3T+6 T^{1-\alpha}\tilde{A}^{2\alpha}}\|z^1-z^2\|_{BMO}.
\end{align}

{\it Step 2. The estimate of $\|Z^{1}-Z^{2}\|_{BMO}$.}
Note that for each $t\in[0,T]$,
\begin{align*}
\overline{Y}^{1}_t-\overline{Y}^{2}_t=\int^T_t(f^1_s-f^2_s)ds-\int^T_t(\overline{Z}^{1}_s-\overline{Z}^{2}_s)dB_s.
\end{align*}
Then applying It\^o's formula to $|\overline{Y}^{1}_t-\overline{Y}^{2}_t|^2$, we have
\[
\|Z^{1}-Z^{2}\|_{BMO}^2\leq 2\sup\limits_{\tau\in\mathcal{T}}
\left\|\mathbb{E}_{\tau}\left[\int^T_{\tau}|\overline{Y}^{1}_s-\overline{Y}^{2}_s||f^1_s-f^2_s|ds\right]\right\|_{\mathcal{L}^{\infty}}.
\]
By the Assumption ($H^q_f$), the inequalities \eqref{12q}, \eqref{13q} and \eqref{15q}, we deduce that for any $h\in (0, \delta^A]$,
\begin{align*}
&\sup\limits_{\tau\in\mathcal{T}}
\left\|\mathbb{E}_{\tau}\left[\int^T_{\tau}|\overline{Y}^{1}_s-\overline{Y}^{2}_s||f^1_s-f^2_s|ds\right]\right\|_{\mathcal{L}^{\infty}}\\
&\leq \lambda\|\overline{Y}^1-\overline{Y}^2\|_{\mathcal{S}^{\infty}}\left(2T\|y^1-y^2\|_{\mathcal{S}^{\infty}}+T\sup_{0\leq t\leq T}|k^1_t-k^2_t|\right)\\
&\quad+\lambda\|\overline{Y}^1-\overline{Y}^2\|_{\mathcal{S}^{\infty}}\sup\limits_{\tau\in\mathcal{T}}
\left\|\mathbb{E}_{\tau}\left[\int^T_{\tau}(1+|\overline{Z}^1_s|+|\overline{Z}^2_s|)|\overline{Z}^{1}_s-\overline{Z}^{2}_s|ds\right]\right\|_{\mathcal{L}^{\infty}}\\
&\quad+\lambda\|\overline{Y}^1-\overline{Y}^2\|_{\mathcal{S}^{\infty}}\sup\limits_{\tau\in\mathcal{T}}
\left\|\mathbb{E}_{\tau}\left[\int^T_{\tau}(1+\mathbb{E}[|z^1_s|^2]^{\frac{\alpha}{2}}+\mathbb{E}[|z^2_s|^2]^{\frac{\alpha}{2}})\mathbb{E}[|z^{1}_s-z^{2}_s|^2]^{\frac{1}{2}}ds\right]\right\|_{\mathcal{L}^{\infty}}\\
&\leq \lambda\|\overline{Y}^1-\overline{Y}^2\|_{\mathcal{S}^{\infty}}\left(2T\|y^1-y^2\|_{\mathcal{S}^{\infty}}+T\sup_{0\leq t\leq T}|k^1_t-k^2_t|+\sqrt{3T+6 T^{1-\alpha}\tilde{A}^{2\alpha}}\|z^1-z^2\|_{BMO}\right)\\
&\quad +\lambda\|\overline{Y}^1-\overline{Y}^2\|_{\mathcal{S}^{\infty}}\sqrt{3+6\tilde{A}^2}\|Z^{1}-Z^{2}\|_{BMO}\\
&\leq \lambda\|\overline{Y}^1-\overline{Y}^2\|_{\mathcal{S}^{\infty}}\left(2T\|y^1-y^2\|_{\mathcal{S}^{\infty}}+T\sup_{0\leq t\leq T}|k^1_t-k^2_t|+\sqrt{3T+6 T^{1-\alpha}\tilde{A}^{2\alpha}}\|z^1-z^2\|_{BMO}\right)\\
&\quad+(12+24\tilde{A}^2)\lambda^2\|\overline{Y}^1-\overline{Y}^2\|^2_{\mathcal{S}^{\infty}}+\frac{1}{4}\|Z^{1}-Z^{2}\|^2_{BMO}
\end{align*}
which together with the previous inequality implies that
\begin{equation}\label{16q}
\|Z^{1}-Z^{2}\|_{BMO} \leq2\lambda\sqrt{1+12\lambda^2+24\lambda^2\tilde{A}^2}\left(2T\|y^1-y^2\|_{\mathcal{S}^{\infty}}+T\sup_{0\leq t\leq T}|k^1_t-k^2_t|+\sqrt{3T+6 T^{1-\alpha}\tilde{A}^{2\alpha}}\|z^1-z^2\|_{BMO}\right).
\end{equation}
Meanwhile, it holds that
\begin{equation*}
K^{1}_t=Y^{1}_0-Y^{1}_t-\int_0^tf^1_sds+\int_0^tZ^{1}_sdB_s
\end{equation*}
and
\begin{equation*}
K^{2}_t=Y^{2}_0-Y^{2}_t-\int_0^tf^2_sds+\int_0^tZ^{2}_sdB_s.
\end{equation*}
Therefore, we have
\begin{align*}
\sup_{0\leq t\leq T}|K^{1}_t-K^{2}_t|&\leq |Y^{1}_0-Y^{2}_0|+\|Y^{1}-Y^{2}\|_{\mathcal{S}^{\infty}}+\sup_{0\leq t\leq T}\mathbb{E}\left[\bigg|\int_0^tf^1_sds-\int_0^tf^2_sds\bigg|\right]\\
&\leq 2\|Y^{1}-Y^{2}\|_{\mathcal{S}^{\infty}}+\mathbb{E}\left[\int_0^T|f^1_s-f^2_s|ds\right]\\
&\leq 2\|Y^{1}-Y^{2}\|_{\mathcal{S}^{\infty}}+\lambda T \left(2\|{y^{1}}-y^{2}\|_{\mathcal{S}^{\infty}}+\sup_{0\leq t\leq T}|k^1_t-k^2_t|\right)\\
  &\quad+\lambda\left(\sqrt{3+6\tilde{A}^2}\|Z^{1}-Z^{2}\|_{BMO}+\sqrt{3T+6 T^{1-\alpha}\tilde{A}^{2\alpha}}\|z^1-z^2\|_{BMO}\right).
\end{align*}
We put the estimates \eqref{15q} and \eqref{16q} into \eqref{14q} and obtain
\begin{align*}
&\|Y^1-Y^2\|_{\mathcal{S}^{\infty}}+\|Z^1-Z^2\|_{BMO}+\sup_{0\leq t\leq T}|K^{1}_t-K^{2}_t|\\
&\leq\hat{A}\lambda\left(2T\|y^1-y^2\|_{\mathcal{S}^{\infty}}+T\sup_{0\leq t\leq T}|k^1_t-k^2_t|+\sqrt{3T+6 T^{1-\alpha}\tilde{A}^{2\alpha}}\|z^1-z^2\|_{BMO}\right)
\end{align*}
where we note
\[
\hat{A}:=4+3C+2\sqrt{1+12\lambda^2+24\lambda^2\tilde{A}^2}\left(1+(1+3C)\lambda \sqrt{3+6\tilde{A}^2}\right).
\]

Now we define
\begin{equation}\label{17q}
\widehat{\delta}^{\tilde{A}}:=\min\left( \frac{1}{4\hat{A}\lambda},\frac{1}{12\hat{A}^2\lambda^2},\left(\frac{1}{24\tilde{A}^{2\alpha}}\hat{A}^2\lambda^2\right)^{\frac{1}{1-\alpha}} ,\delta^{\tilde{A}}\right),
\end{equation}
and it is straightforward to check that for any $T \in (0, \widehat{\delta}^{\tilde{A}}]$,
\begin{align*}
&\|Y^1-Y^2\|_{\mathcal{S}^{\infty}}+\|Z^1-Z^2\|_{BMO}+\sup_{0\leq t\leq T}|K^{1}_t-K^{2}_t|\\
&\leq\frac{1}{2}\left(\|y^1-y^2\|_{\mathcal{S}^{\infty}}+\|z^1-z^2\|_{BMO}+\sup_{0\leq t\leq T}|k^1_t-k^2_t|\right)
\end{align*}
which completes the proof.
\end{proof}

Then we give the proof of Theorem \ref{Thm1q}.

\medskip
\begin{proof}[Proof of Theorem \ref{Thm1q}]
We take $\tilde{A}\geq \tilde{A}_0$ and choose $\widehat{\delta}^{\tilde{A}}$ as \eqref{17q}. For $T\in (0,\widehat{\delta}^{\tilde{A}}]$, define $(Y^0,Z^0,K^0)=(0,0,0)$ and by \eqref{1q}, define $(Y^1, Z^1, K^1):=(Y^{Y^0,Z^0,K^0}, Z^{Y^0,Z^0,K^0}, K^{Y^0,Z^0,K^0})$. By recurrence, for each $i\in \mathbb{N}$, set
\begin{equation}\label{18q}
(Y^{i+1}, Z^{i+1}, K^{i+1}):=(Y^{Y^i,Z^i,K^i}, Z^{Y^i,Z^i,K^i}, K^{Y^i,Z^i,K^i}).
\end{equation}
It follows from Lemma \ref{Lem2q} that there exists $(Y,Z,K)\in \mathscr{B}_{\tilde{A}}$ such that
\begin{equation}\label{19q}
\|{Y}^{n}-{Y}\|_{\mathcal{S}^{\infty}}\longrightarrow 0, \quad\|{Z}^{n}-{Z}\|_{BMO}\longrightarrow 0 \quad {\rm and}\quad \|{K}^{n}-{K}\|_{\mathcal{S}^{\infty}}\longrightarrow 0.
\end{equation}
By a standard argument, we have for each $t\in [0, T]$, in $\mathcal{L}^2$,
$$
\int^T_t f(s, Y^{n-1}_s,\mathbb{P}_{Y^{n-1}_s} ,Z^{n}_s,\mathbb{P}_{Z^{n-1}_s},G_s(K^{n-1}))ds\longrightarrow \int^T_t f(s, Y_s,\mathbb{P}_{Y_s}, Z_s,\mathbb{P}_{Z_s}, G_s(K))ds.
$$
Thus, the triple $(Y, Z, K)$ is a solution to the mean-field BSDE (\ref{my1}) with mean reflection and nonlinear resistance and we only need to prove that the solution $(Y, Z, K)$ is ``flat''. Indeed, it is easy to check that
$$
K_T-K_t =\lim_{n\rightarrow \infty} K^n_T-K^n_t
=\lim_{n\rightarrow \infty}\sup_{t\leq s\leq T} L_s(X^n_s),$$
where
\begin{equation*}
X_t^n := \mathbb{E}_t\left[\xi+\int^T_t f(s, Y^{n-1}_s,\mathbb{P}_{Y^{n-1}_s} ,Z^{n}_s,\mathbb{P}_{Z^{n-1}_s},G_s(K^{n-1}))ds\right].
\end{equation*}
By \eqref{18q} and \eqref{19q}, we deduce
$$\|X^n-X\|_{{\mathcal{S}^{\infty}}}\longrightarrow 0,$$
where
 \begin{equation*}
X_t := \mathbb{E}_t\left[\xi+\int^T_t f(s, Y_s,\mathbb{P}_{Y_s} ,Z_s,\mathbb{P}_{Z_s},G_s(K))ds\right].
\end{equation*} Therefore, $K_T-K_t =
\sup_{t\leq s\leq T} L_s(X_s)$, which implies the ``flatness''. The uniqueness follows immediately from Lemma \ref{Lem2q}. The proof is complete.
\end{proof}
\medskip

We will now construct a global solution to equation \eqref{my1} for arbitrarily large time horizon. To this end, we first observe that any local solution $(Y,Z,K)$ on $[T-h,T]$ of the mean-field BSDE \eqref{my1} with mean reflection has a uniform estimate under an additional growth condition.
\begin{lemma}\label{Lem3q}
Assume that $(H^q_{\xi})-(H^q_f)-(H_G)-(H_{\ell})-(H_L)$ hold, $f(t,y,\mu,z,\nu,k)=f(t,y,\mu,z,\nu)$ does not depend on $k$ and satisfies that $|f(t,y,\mu,0,\nu)|\leq \hat{L}$ for some constant $\hat{L}$, and the mean-field BSDE \eqref{my1} with mean reflection has a local solution $(Y,Z,K)\in \mathcal{S}^{\infty}_{[T-h,T]}\times BMO_{[T-h,T]}\times\mathcal{A}^{\infty}_{D,[T-h,T]}$ on $[T-h, T]$ for some $0<h\leq T$.
Then there exists a constant $\overline{L}$ depending only on $C, L, \hat{L}, \lambda$ and $T$ such that
\[
\|Y\|_{\mathcal{S}^{\infty}_{[T-h,T]}}\leq \overline{L}.
\]
\end{lemma}
\begin{proof}
Note that
\[
(Y_t,Z_t)=(\overline{Y}_t+(K_T-K_t),\overline{Z}_t),  \ \forall t\in [T-h,T],
\]
where $(\overline{Y},\overline{Z})\in\mathcal{S}^{\infty}\times BMO$ is the solution to
the following BSDE on $[T-h,T]$:
\[
\overline{Y}_t=\xi+\int^T_t f(s,Y_s,\mathbb{P}_{Y_s},\overline{Z}_s,\mathbb{P}_{\overline{Z}_s})ds-\int^T_t \overline{Z}_s dB_s.
\]
Similarly to \eqref{4q}, we obtain that for each $t\in [T-h, T]$,
 \begin{align*}
\|Y_t\|_{\mathcal{L}^{\infty}} \leq \|\overline{Y}_t\|_{\mathcal{L}^{\infty}} + \sup_{t\leq s\leq T}L_s(X_s),
  \end{align*}
where
$$
 X_t := \mathbb{E}_t\left[\xi+\int^T_t f(s,Y_s,\mathbb{P}_{Y_s},\overline{Z}_s,\mathbb{P}_{\overline{Z}_s})ds\right].
$$
Therefore, we have
 \begin{align}\label{20q}
 \|Y_t\|_{\mathcal{L}^{\infty}}& \leq \|\overline{Y}_t\|_{\mathcal{L}^{\infty}}
 +(C+1)L+ C\mathbb{E}\left[\int^T_t \left(|f(s,Y_s,\mathbb{P}_{Y_s},0,\mathbb{P}_{\overline{Z}_s})|+\frac{1}{2}\lambda+\frac{3}{2}|Z_r|^2\right)dr\right]\notag\\
 &\leq   \|\overline{Y}_t\|_{\mathcal{L}^{\infty}} +(C+1)L+C\left(\hat{L}+\frac{1}{2}\lambda\right)h+\frac{3}{2}C\|Z\|_{BMO_{[T-h, T]}}^2.
 \end{align}
It follows from a similar technique as in \cite[Proposition 2.1 and Proposition 2.2]{BH1} that
\begin{align}\label{21q}
|\overline{Y}_t|\leq (L+\hat{L}T)e^{\lambda T}:=\overline{L}^1, \ \forall t\in[T-h,T],
\end{align}
and
\begin{align}\label{22q}
\|\overline{Z}\|_{BMO_{[T-h, T]}}^2\leq \frac{1\vee T}{3}\left(1+\frac{4\hat{L}}{\lambda}+2\overline{L}^1\right)e^{3\lambda \overline{L}^1}:=\overline{L}^2.
\end{align}
We derive from the inequalities \eqref{20q}, \eqref{21q} and \eqref{22q}  that for each $t\in [T-h, T]$,
\begin{equation}\label{EqL}
  \|Y_t\|_{\mathcal{L}^{\infty}}\leq \overline{L}^1+(C+1)L+C\left(L+\frac{1}{2}\lambda\right)T+\frac{3}{2}C\overline{L}^2:=\overline{L},
\end{equation}
which is the desired result.
\end{proof}

\medskip
\begin{theorem}\label{Thm2q}
Assume that $(H^q_{\xi})-(H^q_f)-(H_G)-(H_{\ell})-(H_L)$ hold and $|f(t,y,\mu,0,\nu,k)|\leq \tilde{L}$ for some constant $\tilde{L}$. Then for arbitrarily large $T$, the mean-field BSDE \eqref{my1} with mean reflection has a unique deterministic flat solution $(Y,Z,K)\in \mathcal{S}^{\infty}\times BMO\times\mathcal{A}^{\infty}_{D}$ on $[0, T]$. 
\end{theorem}
\begin{proof}
  The proof is divided into two steps.

{\it Step 1.} For given $k\in\mathcal{A}^{\infty}_{D}$, we will show that the following BSDE with mean reflection
\begin{align}\label{my3}
  \begin{cases}
  &Y_t=\xi+\int^T_t f(s,Y_s,\mathbb{P}_{Y_s},Z_s,\mathbb{P}_{Z_s},G_s(k))ds-\int^T_t Z_s dB_s+K_T-K_t;\\
  & \ \ \ \ \ \ \ \ \ \ \ \ \ \ \ \ \ \ \mathbb{E}[\ell(t,Y_t)]\geq 0,
  \end{cases} \  \  \forall t\in[0,T],
  \end{align}
  admits a unique deterministic flat solution $(Y,Z,K)\in \mathcal{S}^{\infty}\times BMO\times\mathcal{A}^{\infty}_D$. In fact, recalling $\bar{L}$ defined in \eqref{EqL} with $\hat{L}=\tilde{L}+\lambda \sup_{0\leq t\leq T}|k_t|$ and defining
   \begin{equation*}
    \overline{A}_0=(4+3C)\overline{L}+\left(1+C\lambda\right)\left(1+\frac{3\overline{L}}{\lambda}\right)e^{9\lambda \overline{L}},
    \end{equation*}
    then by Theorem \ref{Thm1q}, there exists some constant $\overline{h}>0$ depending only on $\overline{L},\lambda$ and $C$ together with $\overline{A}_0$ such that the mean-field BSDE \eqref{my3} with mean reflection admits a unique deterministic flat solution $(Y^1,Z^1,K^1)\in \mathcal{S}^{\infty}_{[T-\overline{h},T]}\times BMO_{[T-\overline{h},T]}\times\mathcal{A}^{\infty}_{D,[T-\overline{h},T]}$ on the time interval $[T-\overline{h},T]$.  Furthermore, it follows from Lemma \ref{Lem3q} that $\|Y^1\|_{\mathcal{S}^{\infty}_{[T-\overline{h},T]}}\leq \overline{L}$. Next we take $T -\overline{ h}$ as the terminal time and apply Theorem \ref{Thm1q} again to find the unique deterministic flat  solution of the mean-field BSDE \eqref{my3} with mean reflection $(Y^2,Z^2,K^2)\in \mathcal{S}^{\infty}_{[T-2\overline{h},T-\overline{h}]}\times BMO_{[T-2\overline{h},T-\overline{h}]}\times\mathcal{A}^{\infty}_{D,[T-2\overline{h},T-\overline{h}]}$ on the time interval $[T-2\overline{h},T-\overline{h}]$. Let us set
    \[
    {Y}_t=\sum\limits_{i=1}^2Y^i_t\mathbf{1}_{[T-i\overline{h},T-(i-1)\overline{h})}+Y^1_T\mathbf{1}_{\{T\}}, \  {Z}_t=\sum\limits_{i=1}^2Z^i_t\mathbf{1}_{[T-i\overline{h},T-(i-1)\overline{h})}+
    Z^1_T\mathbf{1}_{\{T\}}
    \]
    on  $[T-2\overline{h},T]$ and  ${K}_t=K^2_t$ on $[T-2\overline{h},T-\overline{h})$, ${K}_t=K^2_{T-\overline{h}}+K^1_t$ on $[T-\overline{h},T]$. One can easily check that
    $( {Y}, {Z}, {K})\in\mathcal{S}^{\infty}_{[T-2\overline{h},T]}\times BMO_{[T-2\overline{h},T]}\times\mathcal{A}^{\infty}_{D,[T-2\overline{h},T]}$ is a deterministic flat
    solution to the mean-field BSDE \eqref{my3} with mean reflection on $[T-2\overline{h},T]$. By Lemma \ref{Lem3q} again, it yields that $\| {Y}\|_{\mathcal{S}^{\infty}_{[T-2\overline{h},T]}}\leq \overline{L}$.
    
    Furthermore, we repeat this procedure so that we can build a deterministic flat
    solution  $(Y,Z,K)\in\mathcal{S}^{\infty}\times BMO\times\mathcal{A}^{\infty}_{D} $ to the mean-field BSDE \eqref{my3} with mean reflection on $[0,T]$. Moreover, it follows from Lemma \ref{Lem3q} that $\| {Y}\|_{\mathcal{S}^{\infty}}\leq \overline{L}$. The uniqueness of the global solution on $[0, T]$ is inherited from the uniqueness of local solution on each time interval.

{\it Step 2.}  For $k^i\in\mathcal{A}^2_D$, $i=1,2$, let $(Y^i,Z^i,K^i)$ be the unique deterministic flat solution to \eqref{my3} with data $k^i$. It follows from a similar argument as in the proof of Lemma \ref{Lem2q} that there exists some constant $\tilde{h}>0$ depending only on $L,\tilde{L},\lambda$ and $C$ together with $\sup_{0\leq t\leq T}|k^1_t|,\sup_{0\leq t\leq T}|k^2_t|$ such that for any $1\leq i\leq N$, with $N$ being the smallest integer satisfying $N\geq \frac{T}{\tilde{h}}$,
\begin{align*}
  \sup_{(i-1)\varepsilon\leq t\leq i\varepsilon \wedge T}|K^{1}_t-K^{2}_t|^2\leq \frac{1}{2}\left(\sup_{0\leq t\leq T} |k^1_t-k^2_t|^2\right).
  \end{align*}  
Since $K^1,K^2$ are determinsitic, we therefore have 
\begin{align*}
  \sup_{0\leq t\leq T}|K^{1}_t-K^{2}_t|^2\leq \max_{1\leq i\leq N}\left(\sup_{(i-1)\varepsilon\leq t\leq i\varepsilon \wedge T}|K^{1}_t-K^{2}_t|^2\right)\leq \frac{1}{2}\left(\sup_{0\leq t\leq T} |k^1_t-k^2_t|^2\right).
\end{align*}
Thus it follows from stanadard contraction mapping argument that mean-field BSDE \eqref{my1} with mean reflection admits a unique deterministic flat solution $(Y,Z,K)\in \mathcal{S}^{\infty}\times BMO\times\mathcal{A}^{\infty}_D$ on $[0,T]$.
\end{proof}
\section{Application to super-hedging under risk constraint}
In this section, we will prensent an application of our results for the super-hedging problem under risk constraint. We consider a financial market, where an investor can decide at time $t\in[0,T]$ to put an amount $\pi_t$ of this wealth $V_t$ in risky assets, and to allocate his consumption $K_t$, where $K$ is an increasing process with $K_0=0$. Following \cite{QX}, we consider the case that the consumption also has effect on the wealth through certain market mechanics characterized by a given mapping 
\begin{equation*}
  f:[0,T]\times\Omega\times\mathbb{R}\times\mathbb{R}^n\times\mathbb{R}_{+}\rightarrow\mathbb{R}.
\end{equation*} 
As a result, given a consumption-investment strategy $(\pi,K)$, we will assume that the associated wealth process $V^{\pi,K}$ statisfies the following stochastic differential euqation
\begin{equation*}
  dV^{\pi,K}_t=f(t,V^{\pi,K}_t,\pi_t,G(K)_t)dt-dK_t+\sigma_t\pi_tdB_t,
\end{equation*}
where $G$ is a mapping, which may depend on the whole history of the consumption process $K$. Using such consumption-investment strategies $(\pi,K)$, the investor wants to hedge a given contingent claim $\xi\in \mathcal{L}^2(\mathcal{F}_T)$. Moreover, the investor also faces regulatory rules and needs to do risk management by restricting his adimissible strategies. Thus, we call a strategy $(\pi,K)$ is admissible if and only if the associated wealth process $V$ satisfies the following constraint 
\begin{equation*}
  \rho\left(t,V^{\pi,K}_t\right)\leq q_t,~~~0\leq t\leq T,
\end{equation*}
where $\rho$ is a given risk meaure and $q$ is a benchmark. In this situation, the investor will look for the super-hedging price 
\begin{equation*}
  V_0=\inf\left\{V^{\pi,K}_0\in\mathbb{R}:~\text{there exists an adimissible strategy such that}~V^{\pi,K}_T\geq \xi~\text{and}~\rho\left(t,V^{\pi,K}_t\right)\leq q_t,~\forall t\in[0,T] \right\},
\end{equation*}
and associated admissible strategy $(\pi,K)$. 

We will first present an extension of Theorem \ref{Thm2L}, where BSDE with risk measure constraint is studied.
\begin{theorem}
  Let $\rho(t,\cdot):[0,T]\times\mathcal{L}^2\rightarrow\mathbb{R}$ be a family of monotonic and translation invariant risk meaures, which are continuous with time and satisfy 
  \begin{equation*}
    \left|\rho(t,X)-\rho(t,Y)\right|\leq C\mathbb{E}\left[|X-Y|\right],~X,Y\in\mathcal{L}^2(\mathcal{F}_t),~0\leq t\leq T.
  \end{equation*}
  Under the assumption $(H_f)-(H_G)$, if $q$ is a continuous deterministc benchmark and $\xi$ satisfies $\rho(T,\xi)\leq q_T$, then the following BSDE with risk meausure reflection 
  \begin{align*}
    \begin{cases}
    &Y_t=\xi+\int^T_t f(s,Y_s,\mathbb{P}_{Y_s},Z_s,\mathbb{P}_{Z_s},G_s(K))ds-\int^T_t Z_s dB_s+K_T-K_t;\\
    & \ \ \  \rho(t,Y_t) \leq q_t\ \ \ \ \ \ \ \ \ \ \ \ \ \ \int_0^T\left(q_t-\rho(t,Y_t)\right)dK_t=0,
    \end{cases} \  \  \forall t\in[0,T],
    \end{align*}
    admits a unique deterministic flat solution.
\end{theorem} 
\begin{proof}
  The proof follows similarly from the arguments of Lemma \ref{Lem1L}, Theorem \ref{Lem1L} and Theorem \ref{Thm2L} by replacing the map $L_t$ by $\rho(t,\cdot)-q_t$ and the monotoncity and translation invariance property of $\rho(t,\cdot)$.
\end{proof}
\begin{remark}
  We recall that for fixed $t$, a risk meaure is a map $\rho(t,\cdot):\mathcal{L}^2(\mathcal{F}_t)\rightarrow\mathbb{R}$ satisfying $\rho(t,0)=0$ and 
  \begin{itemize}
    \item Monotonicity: $X\leq Y\Rightarrow\rho(t,X)\geq \rho(t,Y)$, for $X,Y\in\mathcal{L}^2(\mathcal{F}_t)$;
    \item Translation invariance: $\rho(t,X+m)=\rho(t,X)-m$, for $X\in\mathcal{L}^2(\mathcal{F}_t)$ and $m\in\mathbb{R}$.
  \end{itemize}
  Typical examples of such risk measures are coherent risk measures (see \cite{ADEH}), for instance 
  \begin{equation*}
    \rho(t,X)=\sup\left\{\mathbb{E}^{\mathbb{Q}}[X]:\mathbb{Q}\in\mathcal{Q}_t\right\},
  \end{equation*}
  where $\mathcal{Q}_t$ is a set of absolutely continuous probabilities with respect to $\mathbb{P}$. $\rho(t,\cdot)$ is Lipschitz as soon as the set of probability densities is bounded, which is satisfied by the classical expected shortfall risk measure. 
\end{remark}
We now go back to our super-hedging problem with risk constraint. Suppose the following conditions:
\begin{itemize}
  \item super-hedging strategy exists,
  \item $q$ is a continuous determinsitic benchmark and $\xi$ satisfies $\rho(T,\xi)\leq q_T$,
  \item $f$ is Lipschitz $G$ satisfies $(H_G)$,
\end{itemize}
are satisfied, if the investor restricts himself to deterministic consumption process, relying our results, the super-hedging price $V_0$ is given by the unique deterministic flat solution to the following BSDE with risk measure reflection
\begin{align*}
  \begin{cases}
  &Y_t=\xi-\int^T_t f(s,Y_s,Z_s,,G_s(K))ds-\int^T_t Z_s dB_s+K_T-K_t;\\
  & \ \ \  \rho(t,Y_t) \leq q_t\ \ \ \ \ \ \ \ \ \ \ \ \ \ \int_0^T\left(q_t-\rho(t,Y_t)\right)dK_t=0,
  \end{cases} \  \  \forall t\in[0,T].
  \end{align*}
  Indeed, $V_0$ coincides with $Y_0$, $(\sigma^{-1}Z,K)$ is the associated strategy.


\begin{thebibliography}{99} \small

\bibitem{ADEH} Artzner, P., Delbaen, F., Eber, J.M. and Heath, D. Coherent measures of risk, Math. Finance 9(3), 203-228, 1999. 
\bibitem{BC} Briand, P., Chaudru de Raynal, P.E.,  Guillin, A. and  Labart, C. Particles systems and numerical schemes for mean feflected stochastic differential equations, Ann. Appl. Probab., 20(4), 1884-1909, 2020.

\bibitem{BH} Briand, P., Elie, R. and Hu, Y.   BSDEs with mean reflection, Ann. Appl. Probab., 28(1), 482-510, 2018.


\bibitem{BH1} Briand, P. and Elie, R. A simple constructive approach to quadratic BSDEs with or without delay,
Stochastic Process. Appl., 123(8), 2921-2939, 2013.

\bibitem{BH0} Briand, P. and Hibon, H. Particles systems for mean reflected BSDEs. Stochastic Process. Appl., 131, 253-275, 2021. 

\bibitem{BH2}  Briand, P. and Hu, Y. BSDE with quadratic growth and unbounded terminal
value, Probab. Theory Related Fields, 136(4), 604-618, 2006.

\bibitem{BH3}  Briand, P. and Hu, Y.  Quadratic BSDEs with convex generators and unbounded terminal conditions,
Probab. Theory Related Fields,  141(3-4), 543-567, 2008.

\bibitem{BDLP} Buckdahn, R., Djehiche, J., Li, J. and Peng, S. Mean-field backward stochastic differential equations: a limit apparoach, Ann. Probab. 37(4), 1524-1525, 2009.

\bibitem{BLP} Buckdahn, R., Li, J. and Peng, S. Mean-field backward stochastic differential equations and related partial differential euqations, Stochastic Process. Appl., 119, 3133-3154, 2009.


\bibitem{HHLLW} Hibon, H., Hu, Y., Lin, Y., Luo, P. and Wang, F. Quadratic BSDEs with mean reflection, Mathematical Control and Related Fields, 8(3\&4), 721-738, 2018.

\bibitem{HHT} Hibon, H., Hu, Y. and Tang, S. Mean-field type quadratic BSDEs, arXiv:1708.08784, 2017.


\bibitem{K} Kazamaki, N.  Continuous exponential martingales and BMO, Lecture Notes in
Mathematics, 1579. Springer-Verlag, Berlin, 1994.

\bibitem{K1} Kobylanski, M. Backward stochastic differential equations and partial differential
equations with quadratic growth, Ann. Probab., 28(2), 558-602, 2000.

\bibitem{PP} Pardoux, E. and Peng, S.  Adapted solution of a backward stochastic differential equation,
Systems Control Lett., 14(1), 55-61, 1990.
\bibitem{QX} Qian, Z. and Xu, M. Reflected backward stochastic differential equations with resistance. Ann. Appl. Probab., 28(2), 888-911, 2018.

\end{thebibliography}
\end{document}